\documentclass[12pt,reqno]{amsart} 
\usepackage{amsmath}
\usepackage{amssymb, latexsym, amsfonts, amscd, amsthm, mathrsfs, enumerate, esint}
\usepackage[usenames,dvipsnames]{color}
\usepackage[all]{xy}
\usepackage{graphicx}
\usepackage{epsfig}
\usepackage{hyperref}
\usepackage{marginnote}
\usepackage{mathtools} 			
\usepackage{bbm}
\usepackage{amssymb}
\usepackage{amsmath}
\usepackage{amsfonts,latexsym,amstext}
\usepackage{marginnote}
\usepackage[numbers,sort]{natbib}
\usepackage{stmaryrd}
\usepackage{soul}

\numberwithin{equation}{section}

\definecolor{purple}{rgb}{0.9,0,0.8}

\definecolor{gray}{rgb}{0.7,0.7,0.7}

\newtheorem{thm}{Theorem}[section]
\newtheorem{cor}[thm]{Corollary}
\newtheorem{lem}[thm]{Lemma}
\newtheorem{ppn}[thm]{Proposition}

\theoremstyle{definition}
\newtheorem{defn}[thm]{Definition}

\newtheorem{rem}[thm]{Remark}

\newcommand{\beq}{\begin{equation}}
\newcommand{\eeq}{\end{equation}}

\newcommand{\D}{\mathcal{D}}

\newcommand{\F}{\mathcal{F}}

\newcommand{\M}{\mathbb{M}}
\newcommand{\N}{\mathbb{N}}

	\renewcommand{\P}{\mathbb{P}}

\newcommand{\R}{\mathbb{R}}

\newcommand{\Z}{\mathbb{Z}}

\DeclareMathOperator{\argmin}{arg\,min}

\renewcommand{\varnothing}{\varnothing}

\renewcommand{\setminus}{\backslash}

\renewcommand{\M}{{\cal M}}

\DeclareMathOperator{\reals}{\mathbb{R}}

\renewcommand{\M}{\mathcal{M}}

\DeclareMathOperator{\constW}{c_1}  
\DeclareMathOperator{\constC}{c_2} 
\DeclareMathOperator{\constR}{c_2} 
\DeclareMathOperator{\constG}{c_2}    
\DeclareMathOperator{\constF}{c_3} 	
\DeclareMathOperator{\constS}{c_4}  
\DeclareMathOperator{\constT}{c_5}  
\DeclareMathOperator{\constZ}{c_6}  
\DeclareMathOperator{\constA}{c_7} 
\DeclareMathOperator{\constB}{c_8} 
\DeclareMathOperator{\constN}{c_9} 
\DeclareMathOperator{\constM}{c_{10}} 
\DeclareMathOperator{\constL}{c_{11}} 
\DeclareMathOperator{\constP}{c_{12}} 
\DeclareMathOperator{\constQ}{c_{13}} 
\DeclareMathOperator{\constX}{c_{14}}  
\DeclareMathOperator{\constH}{c_{15}} 
\DeclareMathOperator{\constI}{c_{16}} 
\DeclareMathOperator{\constJ}{c_{17}} 
\DeclareMathOperator{\constY}{c_{18}}  

\DeclareMathOperator{\constD}{\lambda} 

\definecolor{pink}{rgb}{0.858, 0.188, 0.478}

\title[FPP and Distance Learning]{Nonhomogeneous Euclidean first-passage percolation and distance learning}

\author[P. Groisman]{P. Groisman}
\address{IMAS-CONICET, Departamento de Matem\'atica, Fac. Cs. Exactas y Naturales, Universidad de Buenos Aires, Argentina and NYU-ECNU Institute of Mathematical Sciences at NYU Shanghai}
\email{pgroisma@dm.uba.ar}
\author[M. Jonckheere]{M. Jonckheere}
\address{IMAS-CONICET and Intituto de C\'alculo, Fac. Cs. Exactas y Naturales, Universidad de Buenos Aires.}
\email{mjonckhe@dm.uba.ar}
\author[F. Sapienza]{F. Sapienza}
\address{Aristas S.R.L. and Departamento de Matem\'atica, Fac. Cs. Exactas y Naturales, Universidad de Buenos Aires, Buenos Aires, Argentina. }
\email{f.sapienza@aristas.com.ar}

\begin{document}

\begin{abstract}
Consider an i.i.d. sample from an unknown density function supported on an unknown manifold embedded in a high dimensional Euclidean space.
We tackle the problem of learning a distance between points, able to capture both the geometry of the manifold and the underlying density.
We define such a sample distance and prove the convergence, as the sample size goes to infinity, to a macroscopic one that we call {\em Fermat distance} as it minimizes a path functional, resembling Fermat principle in optics. 
The proof boils down to the study of geodesics in Euclidean first-passage percolation for nonhomogeneous Poisson point processes. 
\end{abstract}



\maketitle

\section{Introduction}

The main motivation for this article is the following problem: 

\begin{quotation}
Let $Q_n= \{q_1, \dots, q_n\}$ be independent random points with common density supported on a Riemannian manifold. Define a distance in $Q_n$ that captures {both} the intrinsic structure of the manifold and the density. 
\end{quotation}

This problem {arises naturally} in tasks like clustering or dimensionality reduction of high-dimensional data, for which the notion of distance between points that is used is crucial. A typical example is the problem of clustering images according to their {visual} content (say, pictures of hand-writing digits). Even for black and white low-resolution pictures, as low as $30\times30$ pixels, the ambient space is {already} $\R^{900}$. Two important considerations in this kind of problems are:	
	\begin{itemize}
	\item \textit{Curse of dimensionality.} Euclidean or Minkowsky distance are not a good choice because in high dimensional spaces every two points of a typical large 
	set are at similar distance \cite{aggarwal01}.
	\item \textit{Data support.} Real data usually lies in a manifold of much smaller dimension. They can be described with a few degrees of freedom, each of these representing one intrinsic variable that parametrize the manifold. In this context, {the} Euclidean distance can be very different from {the} geodesic one, which is more adequate.
	\end{itemize}

{Nevertheless,} considering the geodesic distance {might still not be good enough} since it does not take into account the underlying {density} of the points given by $f$. For example, if $f$ is given by a mixture (with equal weights) of two one-dimensional Gaussian distributions with means 0, 10 and variances $1$ and $2$, respectively, we would like the point 5 to be closer to 10 than to 0. Of course, for a real case scenario the manifold and the density function $f$ are unknown, {but for many learning tasks, it is certainly preferable} to define a distance that takes both into account. 

A fundamental step towards solving this problem was done by Tenenbaum, de Silva and Landford with Isomap, \cite{Tenenbaum00}. This estimator {was shown to achieve better results for} dimensionality reduction tasks by estimating the geodesic distance between points. However, it is independent of the density $f$ from which points are sampled and, as a consequence, it is unable to give/use information about it.  In \cite{CH, CH2} the authors consider sample statistics that capture the intrinsic dimension of the manifold and the intrinsic entropy. They take into account both the manifold structure and the density function. {More general results in the manifold setting are given in \cite{penrose2013limit}.}
Although {their estimators} have a similar flavor to our proposal, they are different in the sense that they consider global properties of the manifold while we are concerned with the distance between pairs of points. {The use of density based distance was explored in several papers,\cite{orlitsky2005estimating, chang2005robust, bijral2012semi, alamgir2012shortest}.} 

{In this article, we elaborate on a distance learning methodology that we introduce in \cite{SGJ}, focusing here mainly on the mathematical aspects of the problem. We define the {\em Fermat distance}, a macroscopic quantity that measures the distance between two points in a manifold in this context, and the {\em Sample Fermat distance} as a distance inferred from the data that estimates the former one. Our contribution is then three-fold:}

\begin{itemize}

\item {{\em Consistency.} We show that a scaled version of the sample Fermat distance convergences almost surely towards the macroscopic Fermat distance, on connected open sets of Euclidean space and as a consequence also on manifolds.}

\item { {\em Convergence of geodesics.}  We show that sample geodesics (minimizers of our microscopic action functional) do converge towards macroscopic minimizers of the macroscopic functional that defines the Fermat distance. The core of the proof is a bound from above for the arc length of geodesics in nonhomogeneous Euclidean first-passage percolation.}

\item  { {\em Complexity.} We show that with large probability the sample Fermat distance can be computed in $\mathcal O (n^2\log^2 n)$ operations by restricting ourselves to ``local'' paths.
}
  
\end{itemize}

{ These fundamental mathematical properties shed light on the potential efficiency of the sample Fermat distance for unsupervized learning tasks.} 
For a more detailed discussion on real applications to distance and manifold learning, clustering, dimensionality reduction and comparison with other methods, computational aspects, etc. we refer to \cite{SGJ} {(see also \cite{Tenenbaum00, CH} on alternative proposals)}. A {practical 
implementation of an algorithm computing our proposed distance}
can be downloaded from \url{http://www.aristas.com.ar/fermat/index.html}.  

As a byproduct of the analysis we obtain a {\em Shape Theorem} for first-passage percolation in nonhomogeneous point processes as in \cite{HN} (see Corollary \ref{cor:shape.theorem} below). In our setting, due to the inhomogeneous space, a huge family of shapes rather than Euclidean ball can be obtained depending on the density of points. 

\subsection*{Related work}
After a preprint version of this paper was uploaded to Arxiv, McKenzie and Damelin \cite{MCD} proposed {independently} similar clustering applications as in \cite{SGJ}, focusing more specifically on spectral clustering. They use the term {\em  power weighted shortest path metric} (p-wspm) for what we call {\em Fermat distance} and consider a slightly different setting. They assume that the data is sampled from a density supported in a disjoint union of manifolds and they use this distance in a different way. While we consider $K-$medoids algorithm for clustering, they perform spectral clustering (note that both algorithms depend strongly on the notion of distance between data points that is used). They also present a fast algorithm for computing these distances. They show in specific synthetic and real-data examples that spectral clustering with Fermat distance outperforms the classical spectral clustering algorithm with Euclidean distance, similar to what we show in \cite{SGJ} with $K-$medoids. On a different context, in \cite{Carpio777169} the use of Fermat distance in combination with Topological Data Analysis techniques is considered to understand the topology of gene expressions in healthy and cancer tissue. By means of persistent homology, the authors are able to distinguish the topology of these two datasets.

\section{Definitions and main results}

Following \cite{HN, HN2}, let $Q$ be a non-empty, locally finite, subset of $\R^d$. We refer to the elements $q \in Q$ as {\em particles}. For any $x \in \R^d$ we denote $q(x)$ the center of the Voronoi cell of $x$ with respect to $Q$. That is, $q(x)$ is the particle closest to $x$ in Euclidean distance. Given $x, y \in \R^d$, a {\em path} from $x$ to $y$ is a finite sequence of particles $(q_1, \dots, q_k)$ with $k\ge 2$, $q_1=q(x)$ and $q_k=q(y)$. The line segment from $x$ to $y$ is denoted $\overline{xy}$ and $\overline{(q_1, \dots, q_k)}$ denotes the polygonal path of line segments $\overline{q_1 q_2}, \, \overline{q_2 q_3}, \dots, \overline{q_{k-1} q_k}$. We also use $|\overline{(q_1, \dots, q_k)}|$ for its arc length, $|x|$ for the Euclidean norm of $x$ and for $a>0$, $C \subset \R^d$, $B(C,a)$ is the set given by
\[
B(C,a)=\bigcup_{z\in C} B(z,a),
\]
where $B(z,a)$ is the open ball centered at $z$ with radius $a$ with respect to the Euclidean norm. {We can now define the {sample Fermat distance}.}

\begin{defn}
For $\alpha \ge 1$ and $x,y \in \R^d$, we define the {\em sample Fermat distance} as
\begin{equation}
D_{Q, \alpha}(x,y)=\inf \left \{ \sum_{j=1}^{K-1}|q_{j+1}-q_j|^\alpha: (q_1, \dots, q_K) \text{ is a path from $x$ to $y$}, \,\, K\ge 1 \right \}.
\label{D_Q} 
\end{equation}
\end{defn}

Notice that {$D_{Q, \alpha}$} {satisfies the} triangular inequality and defines a metric over $Q$ and a pseudometric over $\R^d$. When not strictly necessary we will drop the dependence of all these quantities on $\alpha$ and $Q$. 

This distance was considered previously in \cite{HN, HN2} to construct continuous models of first-passage percolation {and extended in \cite{hwang2016} to the setting in which $Q$ is a set of independent random points with common density $f$ supported on a Riemannian manifold. The key difference between \cite{hwang2016} and our setting is that in \cite{hwang2016} the authors consider intrinsic geodesic distance raised to the power $\alpha$ as weights in \eqref{D_Q} instead of the Euclidean distance, while we keep the last one.
 Their goal is to work in the (more general) intrinsic manifold setting while ours is to define a (computable) estimator of a suitable distance. Although in both cases we get the same limiting object, the microscopic objects are different and so do the proofs, with the exception of the case in which $S$ is an open convex set of $\mathbb R^d$, where geodesic and Euclidean distance do coincide. In that case, our Corollary \ref{cor:n.points} below is indeed contained in \cite{hwang2016}. The Shape Theorem, Corollary \ref{cor:shape.theorem}, the convergence of geodesics, Corollary \ref{cor:convrgence.of.geodesics} and the local nature of geodesics, Proposition \ref{prop:k_nn}, crucial to compute geodesics, are new (even in the case of an open convex set). When $S$ is not convex, or a manifold with dimension strictly smaller than the one of the ambient space, also Corollary \ref{cor:n.points} is new.}

We will focus on (but not restrict {ourselves} to) the case in which $Q$ is a Poisson Point Process (PPP) although the intensity function will be different in different instances. We assume that all the processes involved are constructed in {a probability space} $(\Omega, \F, \P)$.
All the ``almost sure'' statements are with respect to $\P$. We write $Q\sim {\rm Poisson}(S,g)$ when $Q$ is a PPP on $S$ with intensity function $g$ with respect to volume element on $S$. We include here the possibility that $S$ is a manifold with dimension smaller than $d$.  

Notice that for $\alpha=1$ this distance coincides with {the} Euclidean distance but for $\alpha>1$ large {jumps} are discouraged {and this results in a different distance that penalizes paths in which consecutive points are far away to each other}. We also call $r_{Q,\alpha}(x, y)$ the unique path along which $D_{Q,\alpha}(x,y)$ is achieved when it is defined (that is the case a.s. if, for example, $x$, $y$ are deterministic and $Q$ is a PPP, \cite{HN}).

\medskip


Next we define a macroscopic version of the sample Fermat distance that we call the {\em macroscopic Fermat distance} or simply the {Fermat distance} or as follows.

\begin{defn}
For a continuous and positive function $f$, {$\beta \ge 0$} and $x,y \in S$ we define the {\em Fermat distance} $\mathcal D_{f,\beta}(x,y)$ as
\begin{equation}
\label{Fermat.Distance}
\mathcal T_{f, \beta}(\gamma) = \int_\gamma f^{-\beta}, \qquad  \mathcal D_{f,\beta}(x,y) = \inf_\gamma \mathcal T_{f, \beta}(\gamma). 
\end{equation}
Here the infimum is taken over all continuous and rectifiable paths $\gamma$ contained in $\bar S$, {the closure of $S$}, that start at $x$ and end at $y$; and the integral is {understood with} respect to arc length given by Euclidean distance. 
\end{defn}
We will omit the dependence on $\beta$ and $f$ when not strictly necessary. This definition coincides with Fermat Principle in optics for the path followed by light in a non-homogeneous media when the refractive index is given by $f^{-\beta}$. We will call the minimizer $\gamma^\star$ in \eqref{Fermat.Distance} a {\em macroscopic $f$-geodesic} between $x$ and $y$. Observe that $f$-geodesics are likely to lie in regions where $f$ is large.

\subsection{Consistency}

Our main result consists in proving that the sample Fermat distance when appropriately scaled converges to the Fermat distance.
In other words, the scaled sample Fermat distance is a strongly consistent estimator of the macroscopic one.

\begin{thm}
\label{teo:main_theorem_poisson}
Let $S \subset \R^d$ be an open connected set with $C^1$ (or empty) boundary. Let $f\colon \bar S \to [m_f,M_f]$ be a continuous intensity function. Assume $m_f>0$. For each $n\in \N$ let $Q_n \sim {\rm Poisson}(S, nf)$. Given $x,y \in S$ and $\varepsilon >0$, there exist constants $\mu, \constW, \constC$ and $n_0$ such that
		\begin{equation}
		\P \left ( \left | n^\beta D_{Q_n,\alpha} (x, y) - \mu \mathcal D_{f, \beta}(x,y) \right| > \varepsilon \right) \le e^{-\constW n^{\constC}}, 
		\label{eq:main_theorem}
		\end{equation}
		for $\beta = (\alpha-1)/d$  and every $n \ge n_0$. In particular
		\begin{equation*}
		\lim_{n \to \infty} n^\beta D_{Q_n,\alpha} (x, y) = \mu \mathcal D_{f, \beta}(x,y) \quad \text{almost surely.}
		\end{equation*}
\end{thm}
The Poisson assumption can be replaced by assuming $Q$ is an i.i.d. sample with density $f$ using a simple large deviations estimate.
\begin{cor}
\label{cor:n.points}
The same result holds if we replace $Q_n$ by a set of $n$ independent points with common density $f$.
\end{cor}

If we ``invert'' Theorem \ref{teo:main_theorem_poisson} as in \cite[Theorem 1]{HN} we obtain the Shape Theorem. For $t>0$ define
\[
 \mathbb B_n(x,t) = \{ y\in \R^d \colon D_{Q_n,\alpha}(x,y)<t\} \text{ and }  \mathcal B(x,t) = \{ y\in \R^d \colon \mathcal D_{f,\beta}(x,y)<t\}.
\]

\begin{cor}[Shape Theorem]\label{cor:shape.theorem} In the setting of Theorem \ref{teo:main_theorem_poisson}, let $x\in S$, $\varepsilon$ with $0<\varepsilon<\mu^{-1}$ and $t$ such that $\mathcal B(x,(\mu^{-1}+\varepsilon)t) \subset S$. Then, a.s. there is $n_0$ such that for $n\ge n_0$ we have
\[
\mathcal B(x,(\mu^{-1}- \varepsilon)t) \subset n^\beta \mathbb B_n(x,t) \subset  \mathcal B(x,(\mu^{-1}+\varepsilon)t).
\]

\end{cor}

\begin{rem}
If $S$ is not connected and $x$ and $y$ belong to different connected components of $\bar S$, we have $\mathcal D_{f,\beta}(x,y)=\liminf n^\beta D_{Q_n}(x,y)=\infty$ a.s. If $x,y$ belong to the same connected component, we can restrict ourselves to this component. So, the connectedness assumption {can be dropped and is assumed for simplicity.}

If the Euclidean norm in \eqref{D_Q} is replaced by another distance, similar results can be obtained with the line integrals with respect to arc length replaced by line integrals with respect to the distance involved. It could be interesting to explore other choices.

The $C^1$ smoothness assumption for the boundary of $S$ is not really necessary either and can be relaxed up to some point. {For instance, it is enough (but actually  not necessary) to suppose $S$ to be locally convex at points of the boundary where it is not $C^1$.}  Also, if we allow $m_f=0$ (which can be done with some extra work), no regularity assumptions are needed at boundary points where $f$ vanishes. In fact, the only problem one needs to deal with is the case in which the macroscopic geodesic intersects the boundary. This case can be avoided in several ways, but it can certainly happen if $S$ is not convex and $f$ is not negligible at the boundary. We assume in the sequel the stronger $C^1$ assumption to simplify the exposition.
\end{rem}

As a consequence of Theorem \ref{teo:main_theorem_poisson}, we obtain a similar result for points supported on a lower dimensional manifold. We will say that $\M$ is an isometric {$d$-dimensional} $C^1$ manifold embedded in $\R^D$ if there exists $S \subset \R^d$ an open connected set and $\phi : \bar S \to \R^D$ an isometric transformation such that $\phi(\bar S)=\M$. As we mentioned before, in real applications, we expect $d\ll D$, but this is not required.


\begin{thm} 	\label{teo:manifolds}
Assume $\M$ is an isometric $C^1$ $d$-dimensional manifold embedded in $\R^D$ and $f\colon\M \to \reals_{+}$ is a continuous probability density function. Let $Q_n=\{q_1,\dots, q_n\}$ be independent random points with common density $f$. Then, for $\alpha >1$ and $x, y \in \M$ we have  
		\begin{equation*}
		\lim_{n \to \infty} n^\beta D_{Q_n,\alpha} (x, y) = \mu \mathcal D_{f, \beta}(x,y) \quad \text{almost surely.}
		\end{equation*}
Here $\beta = (\alpha-1)/d$ and $\mu$ is a constant depending only on $\alpha$ and $d$; the minimization is carried over all rectifiable curves $\gamma \subset \M$ that start at $x$ and end at $y$. 
\end{thm}


\begin{rem}
{Notice that the scaling factor $\beta = (\alpha-1)/d$ depends on the intrinsic dimension of the manifold, instead of the dimension $D$ of the ambient space.}
\end{rem}

\subsection{Geodesics}

Once, we have the convergence of the distances, it is natural to ask for the convergence of the geodesics. An important step towards this result is to prove that sample geodesic arc length is bounded. This result is not straightforward and follows from geometric arguments combined with large deviations.

\begin{ppn}
\label{prop:bounded.arc.length}
Let $S\subset \R^d$ be a bounded connected open set {with $C^1$ boundary}, $Q_n \sim {\rm Poisson}(S, nf)$ and $\delta >0$. Then, there exists positive constants $\ell, \constG, \constF$ and $ n_0$, with $\constF(\delta)$ depending on $\delta$, such that if $x,y \in S$, $|x-y|>\delta$, then for all $n > n_0$
    \begin{equation}
    \mathbb{P} \left( |\overline {r_{Q_n,\alpha}(x,y)}| > \ell \right) \leq \exp \left( -\constF n^{\constG} \right). 
    \label{eq:bounded_path}
    \end{equation}
As a consequence, we have
    \begin{equation*}
    \limsup_{n \to \infty}  |\overline{r_{Q_n,\alpha}(x,y)}| \leq \ell \qquad \text{almost surely}.
    \end{equation*}
\end{ppn}

{The constant $\constG$ is the same as in Theorem \ref{teo:main_theorem_poisson}.} {Having obtained this upper bound on the arc length of geodesics, we can show that under suitable conditions the microscopic geodesics converge to the macroscopic one.} 

\begin{cor}
\label{cor:convrgence.of.geodesics}
{Let $S\subset \R^d$ be a bounded connected open set and $Q_n \sim {\rm Poisson}(S, nf)$.}
If there is a unique macroscopic $f$-geodesic ${\gamma^\star}$, then $\overline{r_{Q_n,\alpha}(x,y)}$ converges uniformly to ${\gamma^\star}$ {almost surely}.
\end{cor}

\subsection{Complexity}
Finally, we turn our attention to the computability of the sample Fermat distance.  
Computing the minimum in \eqref{D_Q} for every two points in $Q_n$ requires a search in a discrete set of size larger than $n!$. By means of Floyd-Warshall algorithm, this task can be done in $\mathcal O(n^3)$ operations. We prove that we can restrict the search to paths in which each particle $q_i$ of the path is a $k-$th nearest neighbor of $q_{i-1}$ for $k \approx \log n$. Based on this fact, {Dijkstra algorithm} requires $ \mathcal O (n^2\log^2 n)$ operations to compute the distances between every two points in the sample.

Given $k \geq 1$ and $q \in Q_n$, the $k$-th nearest neighbor of $q$, denoted by $q^{(k)}$, is defined by
	\begin{equation*}
	q^{(1)} = \argmin_{ q' \in Q_n \setminus \{ q \} } | q' - q |, \qquad
	q^{(k)} = \argmin_{ q' \in Q_n \setminus \{ q , q^{(1)} , \ldots , q^{(k-1)} \} } | q' - q |  \quad \text{for $k > 1$}.	
	\end{equation*}	  
We use the lexicographic order to break ties. Also denote $\mathcal N_k(z) = \{ q^{(1)}, q^{(2)} , \ldots , q^{(k)} \}$ the set of $k$-nearest neighbors of $q$.
We can now define the {\em  restricted} sample Fermat distance as follows:
\begin{defn}
 For $x, y \in Q_n$,  $\alpha \geq 1$ and $k \in \N$, we define
\[
D_{Q_n}^k (x,y) =  \min \left \{ \sum_{i=1}^{K-1} |q_{i+1} - q_{i} |^\alpha \colon  q_{1}=x,  q_{K}=y,\, q_{i+1} \in \mathcal N_k(q_i),\,  1\le i \le K-1 \right\} .
\]
\end{defn}

We have the following quantitative approximation result:
\begin{ppn}\label{prop:k_nn}
In the setting of Theorem \ref{teo:main_theorem_poisson}, given $\varepsilon > 0$, there exist positive constants $\constS, \constT$ such that if $k > \constS \log ( n / \varepsilon ) + \constT $ we have
	\begin{equation*}
	\P \left (D_{Q_n}^{k}(x,y) = D_{Q_n}(x,y) \right ) > 1-\varepsilon.
	\end{equation*}
\end{ppn}
In other words, with probability at least $1-\varepsilon$, the minimizing path $(q_1, \dots, q_{k_n})$ satisfies  $q_{i+1} \in \mathcal N_{k} (q_{i})$ for every $i = 1, \ldots, {K}-1$.

{While the previous result is certainly an improvement, it might still be unsatisfactory for large data sets.} However, if $n$ is very large, it is possible to appeal to greedy implementations. Given $Q_n$, let us consider a subset {of} {\em landmarks} $\tilde Q \subset Q_n$ with $|\tilde Q| = m$ and $m \ll n$. Then, we compute the minimum path between each of the $m$ landmarks and the rest of the particles in $Q_n$ using Dijkstra's algorithm on the $k$-nearest neighbor graph. This can be done in $\mathcal{O}(m k n \log n)$ operations.  Then, we can bound the exact sample Fermat distance between any two points $q, q' \in Q_n$ by
\[
\max_{\tilde q\in \tilde Q}| D_{Q_n}(q, \tilde q) - D_{Q_n}(q',\tilde q )| \leq D_{Q_n}(q,q') \leq \min_{\tilde q \in \tilde Q} \left( D_{Q_n}(q,\tilde q) + D_{Q_n}(q',\tilde q) \right) ,
\]
see \cite{PBCG}. Notice that the bound from above holds with equality if there is a landmark $\tilde q \in \tilde Q$ in the shortest path between $q$ and $q'$. Due to this fact, an interesting and important problem is to choose a good set of landmarks, \cite{PBCG}. 

\subsection{Organization of the paper}
The rest of the article is organized as follows.
In Section \ref{nonhomogeneous}, we prove several lemmas that lead to the proof of consistency, Theorem \ref{teo:main_theorem_poisson}. Corollary \ref{cor:n.points} can be easily obtained by means of a large deviations principle for Poisson random variables and is left to the reader.
In Section \ref{manifold} we consider the original problem, i.e., the case in which $Q_n$ is a random set of independent points with common density $f$ supported on a manifold and we prove Theorem \ref{teo:manifolds}. 
We then obtain Corollary \ref{cor:convrgence.of.geodesics} as a consequence of Theorem \ref{teo:main_theorem_poisson} after proving that the arc length of microscopic geodesics is bounded, which is done in Section \ref{arclength}. Section \ref{implementation} deals with computational considerations. We show that with large probability $(D_{Q_n}(q, q'))_{q,q' \in Q_n}$ can be computed in $\mathcal O (n^2\log^2 n)$ operations by restricting ourselves to ``local'' paths.

\section{Nonhomogeneous PPP}
\label{nonhomogeneous}

We begin by proving the almost sure convergence of $n^\beta D_{Q_n}(x,y)$ to Fermat distance between $x$ and $y$ for nonhomogeneous PPP stated in Theorem \ref{teo:main_theorem_poisson}. The proof will be split in several lemmas. The first step consists in considering homogeneous PPP in a convex set $S\subset \R^d$. This case has actually been treated in \cite{HN, HN2} where the following is proved.
\begin{ppn}{\cite[Lemma 3 and Lemma 4]{HN}, \cite[Theorem 2.2]{HN2}}
	Assume $S \subset \R^d$ is an open convex set and let $Q_n \sim \text{\rm Poisson}(S,n)$. There exists $0 < \mu < \infty$ such that for any $x, y \in S$ we have
	\begin{equation*}
 	\lim_{n\to \infty} n^{\beta} D_{Q_n} (x, y) = \mu |x - y|, \qquad \text{ almost surely.}
	\end{equation*}
	Moreover, given $\delta > 0$ there exist positive constants $\constD, \constC, \constZ, \constA$, with $\constA$ depending on $\delta$, such that if $| x - y| > \delta$ then 
	\begin{equation*}
	\mathbb{P} \left ( \left | n^\beta D_{Q_n} ( x, y) - \mu | x - y | \right | \geq \constD n^{-1/3d} \right ) 
	\leq 
	\constZ \exp \left( - \constA n^{\constC}  \right).
	\end{equation*}
	for every $n \ge 1$.
	\label{prop:random.number.of.points}
\end{ppn}
The results of \cite{HN, HN2} are proved in fact for the case in which $Q_n$ is replaced by an intensity one PPP and instead of taking $n\to\infty$, the authors consider the limit as $|y| \to \infty$. The adaptation of those results to our setting to get \ref{prop:random.number.of.points} is straightforward by considering the rescaled process $n^{1/d}Q_n$ and using \cite[Theorem 2.4]{HN2} to show that if we have $\tilde Q_n \sim {\rm Poisson}(S,n)$, $\tilde Q_n \sim {\rm Poisson}(\R^d,n)$ and $x,y \in S$, then
\[
 \mathbb P \left(  D_{Q_n} ( x, y) \ne  D_{\tilde Q_n} ( x, y) \right) \leq \constZ \exp \left( - \constA n^{\constC}  \right). 
\]

By means of Proposition \ref{prop:random.number.of.points} we can obtain rough bounds for the nonhomogeneous case.
\begin{lem}
Let $S\subset \R^d$ be an open bounded connected set with $C^1$ (or empty) boundary and $f\colon S \to [m_f,M_f]$ measurable. Assume $m_f>0$. Let $\delta>0$ and $x, y \subset S$ with $|x - y| > \delta$ 
and $Q_n \sim {\rm Poisson}(S,n f)$. Then, for all $\varepsilon > 0$ there exist $n_0 = n_0(\varepsilon)$ and a positive constant $\constB = \constB(\delta)$ such that for all $n > n_0$,	
	\begin{align}
	\mathbb{P} \bigg ( n^\beta D_{Q_n} (x, y)  \leq \mu M_f^{-\beta} | x - y | - \varepsilon  \bigg  )  & \leq \exp \left( -\constB (m_f n)^{\constC} \right), \label{eq:bounds_fmax} \\
	\mathbb{P} \bigg (  n^\beta D_{Q_n} (x, y)  \geq \mu m_f^{-\beta} \D_0(x,y)  + \varepsilon  \bigg )  & \leq \exp \left( -\constB(M_f n)^{\constC} \right).  \label{eq:bounds_fmin}
	\end{align} 
	\label{lema:bounds}
\end{lem}
Here $\D_0(x,y):=\D_{f,0}(x,y)$ is the geodesic distance between $x$ and $y$ defined in accordance with \eqref{Fermat.Distance}.
	
\begin{proof}
Denote ${\rm co}(S)$ the convex hull of $S$. Given two locally finite configurations $Q \subset \tilde Q$, we have $ D_{\tilde Q}(x, y) \leq D_{Q}(x,y)$. Enlarge the probability space to consider two homogeneous PPP $Q_n^-\sim{\rm Poisson}(S,nm_f)$ and $Q_n^+\sim {\rm Poisson}({\rm co}(S),nM_f)$, coupled with $Q_n$ {(see for instance \cite[Section 3.2.2]{moller})} to guarantee that $Q_n^- \subset Q_n \subset Q_n^+$. Then
	\begin{align*}
	\mathbb{P} \bigg ( n^\beta D_{Q_n} (x, y) & \leq \mu M_f^{-\beta} | x - y | - \varepsilon  \bigg  ) 
	 \leq
	\mathbb{P} \bigg ( n^\beta D_{Q_n^+} (x, y)  \leq \mu M_f^{-\beta} | x - y | - \varepsilon  \bigg  )
	\end{align*}		
	Choosing  $n_0$ such that $\varepsilon > \constD (n_0m_f)^{-1/3d}$, by means of Proposition \ref{prop:random.number.of.points} we get\eqref{eq:bounds_fmax}. To prove \eqref{eq:bounds_fmin} we proceed similarly, but we need to be more careful. Since $S$ is open and connected, we can consider a polygonal $\gamma = \overline{(x_0, \dots, x_k)} \subset S$ from $x$ to $y$ with 
\[
|\overline{(x_0, \dots, x_k)}| < \D_0(x,y) + \frac{m_f^\beta \varepsilon}{2\mu} \quad \text{ and } \quad B\left (\frac{x_{i+1}+x_i}{2},{|x_{i+1}-x_i|}\right) \subset S,
\]
for every $0\le i \le k-1$. 
{We claim that $k$ can be taken uniformly bounded for every two points $x,y\in S$. To see that, we proceed by contradiction. Fix one point $z \in S$ and assume there is a sequence of points $z_n \in S$ with the property that any polygonal from $z$ to $z_n$ contained in $S$ is composed by at least $n$ line segments. Since $\bar S$ si compact we can extract a convergent subsequence $z_{n_j} \to z^\star \in \bar S$. If $z^\star \in S$, there is ball centered at $z^\star$ contained in $S$ with points $z_{n}$ for large $n$. Then we can easily construct polygonals contained in $S$ from $z$ to $z_n$ with a bounded number of line segments, a contradiction. Then it should hold that $z^\star \in \partial S$ but since $\partial S$ is $C^1$ we can proceed in the same way to obtain again a contradiction.}

Denote 
\[
Q_{n}^i=Q_n^- \cap  B\left (\frac{x_{i+1}+x_i}{2},{|x_{i+1}-x_i|}\right).
\]
Proceeding as before, we get for every $i$,
\begin{equation*}
 	\mathbb{P} \bigg ( n^\beta D_{Q_n} (x_{i},x_{i+1} )  \geq \mu m_f^{-\beta} | x_{i+1} - x_i | + \varepsilon  \bigg )
	 \leq 
	\mathbb{P} \bigg ( n^\beta D_{Q_n^i} (x_i, x_{i+1})  \geq \mu m_f^{-\beta} | x_{i+1} - x_i | + \varepsilon  \bigg ).
\end{equation*}
Then, 
\begin{align*}
 	\mathbb{P} \bigg ( n^\beta D_{Q_n} (x,y)  \geq \mu m_f^{-\beta} \D_0(x,y) + \varepsilon  \bigg ) & \leq  
 	\mathbb{P} \bigg ( n^\beta \sum_{i=0}^k D_{Q_n} (x_{i},x_{i+1} )  \geq \frac{\mu}{m_f^{\beta}} |\overline{(x_0, \dots, x_k)}| + \frac{\varepsilon}{2}  \bigg )\\
 	& \le \sum_{i=0}^k \mathbb{P} \bigg ( n^\beta D_{Q_n^i} (x_i, x_{i+1})  \geq \mu m_f^{-\beta} | x_{i+1} - x_i | + \frac{\varepsilon}{2k}  \bigg ).
\end{align*}
Using again Proposition \ref{prop:random.number.of.points} we get \eqref{eq:bounds_fmin}.	
\end{proof}

The second step is to show that the distance between consecutive particles in the optimal path vanishes as $n\to \infty$.

	\begin{lem}
	In the setting of Theorem \ref{teo:main_theorem_poisson}, let $(q_1, \dots, q_{k_n})$ be the minimizing path. Given $\delta > 0$, there exists a positive constant $\constN$ such that 
	\begin{equation*}
	\mathbb P \bigg ( \max_{i<k_n} | q_{i} - q_{i+1} | > \delta   \bigg ) \leq   \exp \left( - \constN n \right ).
	\end{equation*}	
	\label{lema:spacing}
	\end{lem}
	
	\begin{proof}
	For any two consecutive points $q_i, q_{i+1}$ in the optimal path we have 
	\begin{equation*}
	Q_n \cap \left \{ z \in S : |z - q_{i+1} |^\alpha + | z - q_i |^\alpha < | q_{i+1} - q_i |^\alpha \right \} = \emptyset.
	\end{equation*}
	Observe that we can choose $\kappa_1$ depending only on $\alpha$ and $d$ such that the region 
	$\{ z \in S : |z - q_{i+1} |^\alpha + | z - q_i |^\alpha < | q_{i+1} - q_i |^\alpha \}$
	contains a cube of edge size $\kappa_1|q_{i+1} - q_i|$. Consider a family $\mathcal C$ of cubes of edge size $\kappa_1 \delta/2$ with vertices in $( \kappa_1 \delta/2) \mathbb Z^d$.
	
	 Notice that the number of cubes in this family that intersect $S$ is finite.
	 Each of these cubes has no particles with probability bounded by { $\exp(-\constM n)$}. If $\max_{i<k_n} | q_{i} - q_{i+1} | > \delta$, then there is a cube in $S$ with side $\kappa_1 \delta$. Such a cube must contain a cube in $\mathcal C$.
	\end{proof}
	Next, we prove that in order to find the optimal path between $x$ and $y$ we can restrict ourselves to certain neighborhoods of any path $\gamma_{xy} \subset S$ that starts at $x$ and ends at $y$. This fact will be used both for points that are close to each other as well as for points that are at a large distance. Denote, $m_f^{\gamma}= \inf\{f(z)\colon z \in B(x, 2|\gamma|) \}$.
	\begin{lem}
	In the setting of Theorem \ref{teo:main_theorem_poisson}, given $\delta >0$, there exist positive constants $\constL$ and $n_0$ such that for every $x,y \in S$ with $|x-y|>\delta$ and a path $\gamma \subset S$ from $x$ to $y$ we have, 
	\begin{equation}
	\mathbb{P} \bigg ( D_{Q_n}(x, y) \neq D_{Q_n \cap B(x,\tilde a|\gamma|)} (x, y)  \bigg ) \leq  \exp \left( - \constL n^{\constC} \right),
	\label{eq:entornos_geo}
	\end{equation}
for every $n > n_0$ and $\tilde a=3 \left({M_f}/{m_f^{\gamma}} \right)^\beta$. In particular,
\begin{enumerate}
 \item[(i)] if $S$ is bounded 
 	\begin{equation*}
	\mathbb{P} \bigg ( D_{Q_n}(x, y) \neq D_{Q_n \cap B(x,a|\overline{xy}|)} (x, y)  \bigg ) \leq \exp \left( - \constL n^{\constC} \right).
	\end{equation*}
	with
	$$ a = \tilde a \sup_{| z - w | \geq \delta } \frac{\D_0(z,w)}{|z - w|} < \infty. $$

 \item[(ii)] if $\overline{xy}\subset S$, we have
	\begin{equation*}
	\mathbb{P} \bigg ( D_{Q_n}(x, y) \neq D_{Q_n \cap B(x,\tilde a|\overline{xy}|)} (x, y)  \bigg ) \leq \exp \left( - \constL n^{\constC} \right).
         \end{equation*}
\end{enumerate}
	\label{lemma:entornos}
	\end{lem}
	\begin{proof}
	Let $z \not \in B(x,a|\gamma|)\cap S$. Given $\delta_1 < \mu (m^{\gamma}_f)^{-\beta} |\gamma| / 3$, consider the events 
	\begin{align*}
	A_n^z &= \bigg \{ n^\beta D_{Q_n} (x, z) \leq
	n^\beta D_{Q_n} (x, y) + \delta_1 \bigg \} \\
	E_n^z &= \bigg \{ n^\beta D_{Q_n}(x, z) \geq  \mu M_f^{-\beta} | x - z | - \delta_1  \bigg \}  \\
	F_n &= \bigg \{ n^\beta D_{Q_n \cap B(x, 2 |\gamma|)}(x, y) \leq \mu (m_f^{\gamma})^{-\beta} |\gamma| + \delta_1 \bigg \}.	
	\end{align*}
	In $A_n^z \cap E_n^z \cap F_n$ we get
	
	\begin{align*}
	\mu M_f^{-\beta} | x - z | 
	& \leq \, 
	n^\beta D_{Q_n} (x, z) + \delta_1 
	 \leq \, 
	n^\beta D_{Q_n} (x , y) + 2 \delta_1 \\
	& \leq \,
	n^\beta D_{Q_n \cap B(x, 2 |\gamma|)} (x , y) + 2 \delta_1
	 \leq \,
	\mu (m_f^{\gamma})^{-\beta} |\gamma| + 3 \delta_1\\ 
	& < \, 
	2 \mu (m_f^{\gamma})^{-\beta} |\gamma|.
	\end{align*}
Since $z \not \in B(x,a|\gamma|)$ implies $| x - z | > a |\gamma|$ and $a=3 \left({M_f}/{m_f^{\gamma}} \right)^\beta$,
	we have 
	{$A_n^z \cap E_n^z \cap F_n = \emptyset$}.
	By Lemma \ref{lema:bounds} there exist $\constB(\delta)$, $n_0(\delta)$ independent of $z$ and a positive constant $\constC$ such that 
	\begin{equation*}
	\mathbb{P} (A_n^z) \leq \mathbb{P}((E_n^z)^c) + \mathbb{P}(F_n^c) \leq 2 \exp \left(  -\constB (m_f n)^{\constC} \right)  \quad \text{for all } n > n_0. 
	\end{equation*}	
	Assume $D_{Q_n}(x, y) < D_{Q_n \cap B(x,a|\gamma|)}(x, y)$ and $\{ \max_{i<k_n} | q_{i} - q_{i+1} | < a |\gamma| \}$. Then there is a particle $q \in Q_n \cap B(x,a|\gamma|)^c \cap B(x,2a|\gamma|)$ with
	\begin{equation*}
	D_{Q_n}(x, y) = D_{Q_n} (x, q) + D_{Q_n} (q , y) \geq D_{Q_n} (x, q).
	\end{equation*}
        Consider the following covering
    	\begin{equation*}
	S \cap \left(B(x,2a|\gamma|) \smallsetminus	B(x,a|\gamma|) \right) 
	\subset 
	\bigcup_{v \in \mathcal V} B \left( v, \delta_0 n^{-1/d} \right).
	\end{equation*}
	Here $\mathcal V \subset S \smallsetminus B(x,a|\gamma|)$ is a finite set of points that can be chosen in such a way that $\#\mathcal V \le \kappa_2 n$ for some constant $\kappa_2$ and {$(2\delta_0)^\alpha < \delta_1$}.
	 Let $v_{q} \in \mathcal V$ be such that $q \in B \left( v_{q} , \delta_0 n^{-1/d} \right)$. If $q$ is the closest particle in $Q_n$ to $v_q$, then $D_{Q_n}(q, v_q)=0$. If that is not the case, there is another particle in $B \left( v_{q} , \delta_0 n^{-1/d} \right)$ and consequently we have $D_{Q_n}(q, v_q) < (2 \delta_0 n^{-1/d})^\alpha$.
	From triangular inequality we get
	\begin{equation*}
	n^\beta D_{Q_n} (x, q) \geq
	n^\beta D_{Q_n} (x , {v_q}) - n^\beta D_{Q_n} (q , {v_q}) \geq 
	n^\beta D_{Q_n} (x , {v_q}) - \delta_1n^{-\alpha/d} 		\geq 
	n^\beta D_{Q_n} (x , {v_q}) - \delta_1.
	  \end{equation*}		
	Hence
	\begin{align*}
	\mathbb{P} \bigg ( D_{Q_n}(x, y) \neq &D_{Q_n \cap B(x,a|\gamma|)} (x, y) \, , \, {\max_{i<k_n} | q_{i} - q_{i+1} | < a |\gamma|}  \bigg )\\
	&\leq  
	\mathbb{P} \bigg ( \exists v \in \mathcal V : n^\beta D_{Q_n}(x, y) \geq n^\beta D_{Q_n}(x , v) - \delta_1  \bigg) \\
	& \leq
	\sum_{v \in \mathcal V} \mathbb{P} \left( (A_n^v)^c \right)  \\
	& \leq 
	2  \kappa_2 n \exp \left(  -\constB (m_f n)^{\constC} \right) \quad \forall n > n_0.   
 	\end{align*}
From Lemma \ref{lema:spacing} and the fact that $\constC<1/d{\le}1$ (\cite{HN2}), we get \eqref{eq:entornos_geo}.
	\end{proof}
 

We are ready to prove the upper bound in \eqref{eq:main_theorem}.

	
\begin{lem}[Upper bound] In the setting of Theorem \ref{teo:main_theorem_poisson}, there are positive constants $\constP$ and $n_0$ such that
\[
 \mathbb{P} \left(  n^\beta D_{Q_n}(x, y) > \mu \mathcal D_{f, \beta}(x,y)  + \varepsilon \right) \le {\rm exp}(-\constP n^{\constC}).
\]
for all $n > n_0$.
\end{lem}

\begin{proof}
	Let $\gamma^\star \subset S$ be a continuous and rectifiable curve {that starts at $x$ and ends at $y$} and such that $\int_{\gamma^\star} \frac{1}{f^\beta} < \mathcal D_{f, \beta}(x,y) + \varepsilon/(4\mu)$.
	 If $\varepsilon < 1$, the arc length $| \gamma^\star |$ is bounded above by
	\begin{equation*}
	| \gamma^\star | <  \ell^\star  := M_f^\beta \left(  \mathcal D_{f, \beta}(x,y) + \frac{1}{4\mu} \right).
	\end{equation*}		
	Let us consider a finite set of points $z_1, z_2 , \ldots , z_M \in \gamma^\star$ sorted according to a parametrization of $\gamma^\star$ that starts at $x$ and ends at $y$, such that $z_1 = x$, $z_M = y$ and $ \delta  < | z_{i+1} - z_i | < 2\delta$. Notice that $M = M(\delta) < \ell^\star / \delta $. Let $\gamma^\star_i$ be the part of $\gamma^\star$ that connects $z_i$ and $z_{i+1}$. Then 
	\begin{equation*}
	\int_{\gamma^\star} \frac{1}{f^\beta} = \sum_{i=1}^{M-1} \int_{\gamma_i^\star}  \frac{1}{f^\beta}.
	\end{equation*}
	Since $f^{-\beta}$ is integrable and uniformly continuous in $\overline{S\cap B(x, a|\gamma|)}$, we can choose $\delta>0$ such that
	\begin{enumerate}
	 \item[(i)] 
	$\displaystyle\sum_{i=1}^{M-1} {\left (\min_{\gamma_i^\star} f \right )^{-\beta} } | z_i - z_{i+1} | < \int_{\gamma^\star} \frac{1}{f^\beta} + \frac{\varepsilon}{4},$
	\item[(ii)] $|z-z'|<\delta \Longrightarrow |f^{-\beta}(z) - f^{-\beta}(z')| < \varepsilon_2:= \varepsilon m_f^\beta / (4\mu \ell^\star)$.
	\smallskip
	\item[(iii)] {${\rm co}(B(\gamma^\star_i,\delta)) \subset S$.}
	\end{enumerate}
        Recall here that {\rm co}$(B)$ denotes the convex hull of $B$. For $i = 1, 2, \ldots, M-1$ consider the set $C_i = {\rm co}(B(\gamma^\star_i,\delta))$. On the one hand,
	\begin{equation*}
	D_{Q_n}(x, y) 
	\leq 
	D_{Q_n \cap \left( \cup_{i=1}^{M-1} C_i \right) }(x, y) 
	\leq	
	\sum_{i=1}^{M-1} D_{Q_n \cap C_i} (z_{i}, z_{i+1}).
	\end{equation*}
	On the other hand, we have 
	\begin{align}
	\mu \mathcal D_{f, \beta}(x,y) + \varepsilon 
	& > 
	\mu \int_{\gamma^\star} \frac{1}{f^\beta} + \frac{3\varepsilon}{4} \nonumber \\
	& > 
	\mu \sum_{i=1}^{M-1} \left (\min_{\gamma^\star_i} f \right )^{-\beta} |z_{i+1} - z_{i} | + \frac{\varepsilon}{2} \nonumber \\
	& \geq
	\mu \sum_{i=1}^{M-1} {\left ( \min_{C_i} f \right )^{-\beta} } |z_{i+1} - z_{i} | + \frac{\varepsilon}{2} - \frac{\mu M \delta}{m_f^\beta} \varepsilon_2 \nonumber \\
	& >
	\mu \sum_{i=1}^{M-1} {\left (\min_{C_i} f \right )^{-\beta} } |z_{i+1} - z_{i} | + \frac{\varepsilon}{4} \nonumber.
	\end{align}
Then,
	\begin{align*}
	\mathbb{P} \Big(  n^\beta D_{Q_n}(x, y) & \geq \mu \mathcal D_{f, \beta}(x,y)  + \varepsilon \Big) 
	 \leq \\
%
        & \leq 	\mathbb{P} \left( \sum_{i=1}^{M-1} n^\beta D_{Q_n \cap C_i}(z_i, z_{i+1}) \geq \mu \sum_{i=1}^{M-1} {\left ( \min_{C_i} f \right )^{-\beta} } |z_{i+1} - z_{i} | + \frac{\varepsilon}{4}  \right) \nonumber \\
	& \leq 	\sum_{i=1}^{M-1} \mathbb{P} \left( n^\beta D_{Q_n \cap C_i}(z_i, z_{i+1}) \geq \mu {\left ( \min_{C_i} f \right )^{-\beta} } |z_{i+1} - z_{i} | + \frac{\varepsilon}{4M}   \right) \nonumber \\
	& \leq 	M \exp \left( -\constB (m_f n)^{\constC} \right) \quad \text{ for all } n > n_0 \nonumber,
	\end{align*}	
	by Lemma \ref{lema:bounds} (applied to each $C_i$). Notice that the constant $\constB$ depends only on $\delta$. This finishes the proof of the lemma. 
	\end{proof}

	\begin{lem}[Lower bound] In the setting of Theorem \ref{teo:main_theorem_poisson}, there exist positive contants $\constQ$ and $n_0$ such that 
	 { \[
	  \mathbb{P} \left(  n^\beta D_{Q_n}(x, y) < \mu \mathcal D_{f, \beta}(x,y)  - \varepsilon \right) \le  {\rm exp}(-\constQ n^{\constR}), 
	 \]}
	 for all $n > n_0$.
	\end{lem}

\begin{proof}	
	
	By Lemma \ref{lemma:entornos} we can assume $S$ is bounded (if it is not bounded, we consider $S\cap B(x, a|\gamma|)$, with $\gamma$ any path from $x$ to $y$ instead of $S$). Let $\gamma_n=(q_1, \dots, q_{k_n})$ be the minimizing path. For $\delta > 0$, consider the event $E_n = \{ \max_{j<k_n} | q_{j} - q_{j+1} | < \delta \}$. If $E_n$ occurs, there are particles $q^\star_1, q^\star_2, \ldots, q^\star_k \in \gamma_n \cap Q_n$ with $ \delta < | q^\star_{i+1} - q^\star_i | < 4\delta$ for $i = 0, 1, 2, \ldots , k$, with $q^\star_0 = x$, $q^\star_{k+1} = y$. We can construct this sequence inductively as follows. Denote $\tau_0=0$,  $q_0^\star=x$. {For $i \ge 0$}, if $| q^{\star}_i - y | < 4\delta$, then $q^\star_{i+1} = y$ and we set $k=i+1$. If not, we choose $q_{i+1}^\star=q_{\tau_i}$ with $\tau_{i+1}=\min \{j>\tau_i \colon  2\delta <  |q_j - q_i^\star| < 3\delta\}$. The existence of $\tau_{i+1}$ (in case we need to define it) is guaranteed since we are assuming that $E_n$ occurs. With this construction we have $|q_k^\star -q_ {k-1}^\star|= |y- q_{k-1}^\star|> |y-q_{k-2}^\star|-|q_{k-1}^\star - q_{k-2}^\star|>4\delta - 3\delta=\delta$ and hence $\delta<|q_{i+1}^\star - q_ i^\star|<4\delta$ for every $1\le i \le k-1$.
%
	We will see that there exists a constant $K$ such that $k \leq K$ with overwhelming probability. This would be immediate if we assume that the arc lengths of the minimizing paths are bounded (which is proved in Section \ref{arclength}), but this assumption is not really necessary at this point as the following argument shows. Notice that
	\begin{equation}
	D_{Q_n} (x , y ) = 
	\sum_{i=0}^{k} D_{Q_n} (q^\star_i, q^\star_{i+1}). 
	\label{eq:separo_en_u}
	\end{equation}
	For $\delta_0 > 0$, that will be chosen later, consider the following covering of $\bar S$,
	\begin{equation*}
	\bar S \subset \bigcup_{v \in \mathcal V} B \left( v, \delta_0 n^{-1/d} \right) .
	\end{equation*}
	Here $\mathcal{V} \subset S$ 
	is chosen such that $\# \mathcal V \leq \kappa_3 n$ for some constant $\kappa_3<\infty$. Let $ w_1, w_2, \ldots, w_k \in \mathcal V$ be such that $q^\star_i \in B(w_i , \delta_0 n^{-1/d} )$ for every $i \leq k$. For a given $i \leq k$ it holds 
	\begin{align*}
	n^\beta D_{Q_n} (q^\star_i , q^\star_{i+1}) 
	& \geq 
	n^\beta ( 
	D_{Q_n} (w_i, w_{i+1})
	- D_{Q_n} (w_i, q_i^\star) 
	- D_{Q_n} (w_{i+1} , q_{i+1}^\star ) 
	)\\
	& \geq n^\beta 
	D_{Q_n} (w_i, w_{i+1}) - 2(2\delta_0)^\alpha. 
	\end{align*} 
	If in addition $\delta_0 < \delta / 4$, we have
	$$| w_i - w_{i+1}| > | q_i^\star - q_{i+1}^\star | - | w_i - q_i^\star | - |w_{i+1} - q_{i+1}^\star| > \delta - \delta /4 - \delta/4 = \delta/2.$$
	Let $\Delta = \mu M_f^{-\beta} \delta / 8$ and choose $\delta_0$ with $2(2\delta_0)^\alpha < \Delta$. Then
	\begin{equation*}
	\mathbb{P} \bigg( \min_i n^\beta D_{Q_n}(q^\star_{i}, q^\star_{i+1}) < \Delta \bigg)
	\leq 
	\mathbb{P} \bigg( \exists \, v_1, v_2 \in \mathcal V \colon |v_1 - v_2| > \delta/2 \text{ and } n^\beta D_{Q_n}(v_1, v_2) < 2 \Delta \bigg).
	\end{equation*}	
	Since the number of possible elections of $v_1$ and $v_2$ is upper bounded by $(\kappa_3 n)^2$, from Lemma \ref{lema:bounds} we conclude that
	\begin{equation*}
	\mathbb{P} \bigg( \min_i n^\beta D_{Q_n}(q^\star_{i}, q^\star_{i+1}) < \Delta \bigg) < (\kappa_3 n)^2 \exp(-\constB (m_f n)^{\constC}).
	\end{equation*}		
	If $n^\beta D_{Q_n}(x,y) < 2\mu m_f^{-\beta} \D_0(x,y)$ and $n^\beta D_{Q_n}(q_i^\star , q_{i+1}^\star) > \Delta$ for every $i \le k$, then from \eqref{eq:separo_en_u} we obtain $k\Delta < 2\mu m_f^{-\beta} \D_0(x,y)$. Hence, for $K=K(\delta):=16 \delta^{-1}(M_f/m_f)^\beta  \D_0(x,y)$,
	\begin{equation}
	\label{eq:k.less.than.K}
	 \P \left(k > K \right) \le \exp(-\constB(M_f n)^{\constC}) +{(\kappa_3 n)^2 \exp(-\constB (m_f n)^{\constC})},
	\end{equation}
for $n$ large enough by \eqref{eq:bounds_fmin}.

 If we choose $ (2\delta_0)^\alpha  < (\varepsilon / 4K)$, using triangular inequality in \eqref{eq:separo_en_u} we get
	\begin{align*}
	n^\beta D_{Q_n} (x , y) 
	& \geq 
	\sum_{i=0}^k 
	n^\beta ( 
	D_{Q_n} (w_i, w_{i+1})
	- D_{Q_n} (w_i, q_i^\star) 
	- D_{Q_n} (w_{i+1} , q_{i+1}^\star ) 
	)\\
	& \geq 
	\sum_{i=0}^k n^\beta (
	D_{Q_n} (w_i, w_{i+1})
	{ - 2 (2\delta_0)^\alpha n^{-\alpha / d}) }\\
	& \geq \sum_{i=0}^k n^\beta 
	D_{Q_n} (w_i, w_{i+1}) - {\varepsilon}/{2}.
	\end{align*} 
	Then,
	\begin{align*}
	\mathbb{P} (  n^\beta & D_{Q_n}(x, y)  \leq \mu \mathcal D_{f, \beta}(x,y)  - \varepsilon )\\ 
	 \leq 
	&  \,\, \mathbb P \bigg(  \exists \, v_1, \ldots, v_k \in \mathcal V \text{ with $k \leq K$ and } \frac{\delta}{2} < |v_i - v_{i+1}| < 5\delta \text{ such that} \nonumber \\
	 &  \sum_{i=0}^k n^\beta 
	D_{Q_n} (v_i, v_{i+1}) \leq \mu \mathcal D_{f, \beta}(x,y) - \frac{\varepsilon}{2} \text{, } E_n 
	\bigg)  + 
	\mathbb P \left( k>K \right) + 
	\mathbb P \left( E_n^c \right).
	\end{align*}
The second term is bounded by \eqref{eq:k.less.than.K} and  Lemma \ref{lema:spacing} gives us an exponential bound for the third one. Let us focus on the first one. Notice that the number of paths $(v_1, v_2, \ldots, v_k)$ with $v_i \in \mathcal V$ and $k \leq K$ is bounded above by $(\kappa_3 n)^{K}$. Fix any one these paths and denote 
	\[
	 M_{f,i}:=\sup_{z \in B({v_i}, a { | \overline{v_iv_{i+1}} | } )\cap S}  f(z)
	\]
	and consider the events
	\begin{align*}
	A_i &= \left \{ D_{Q_n}(v_i, v_{i+1}) = D_{Q_n \cap B({v_i}, a{ | \overline{v_iv_{i+1}} | } ) }(v_i, v_{i+1}) \right \} \\
	B_i &= \left \{ n^\beta D_{Q_n \cap B({v_i}, a { | \overline{v_iv_{i+1}} | }) }(v_i, v_{i+1}) \geq \mu {M_{f,i}^{-\beta}}   |v_i - v_{i+1} | - \frac{\varepsilon}{8K} \right \} \cap A_i.
	\end{align*}
	From Lemma \ref{lema:bounds} and Lemma \ref{lemma:entornos} we get that 
	\begin{equation*}
	\mathbb P (B_i^c) \leq  \exp \left( -\constB m_f^{\constC} n^{\constC} \right) + \exp \left( -\constL n^{\constC} \right) \quad \forall n>n_0, \quad i = 1, 2, \ldots, k-1. 
	\end{equation*}
	The constants $\constB$, $\constL$ and $n_0$ depend on $\delta$.
	Now choose $\delta > 0$ such that for $z,z' \in S$ with  $| z - z' | < 5(a+1)\delta$  implies $| f^{-\beta}(z) - f^{-\beta}(z') | < \varepsilon_3 = {\varepsilon m_f^{2\beta} / (128 \mu \D_0(x,y )M_f^\beta) }$. Denote $r_i$ the geodesic between $v_i$ and $v_{i+1}$. We have, 
	\begin{equation*}
	\sum_{i=0}^k 
	M_{f,i}^{-\beta} 
	|v_i - v_{i+1}|
	>
	\sum_{i=0}^k
	{(1-\varepsilon_3)}{\left (\min_{r_i} f \right)^{-\beta} } 	
	| v_i - v_{i+1} |
	> 
	\sum_{i=0}^k {\left (\min_{r_i} f \right)^{-\beta} } 	
	| v_i - v_{i+1} | - \frac{\varepsilon}{8\mu}
	\end{equation*}
	Since the boundary of $S$ is $C^1$, we can control the geodesic distance by the Euclidean distance uniformly in $\bar S$. More precisely, for each $x \in S$ there exists $\delta_x > 0$ such that $B(x, \delta_x) \subset S$ and consequently $\D(x, y)=|x - y|$ for all $y \in B(x, \delta_x)$. If $x \in \partial S$, since the boundary of $S$ is $C^1$, we have $\D_0(x,y)= |x-y| + o(|x-y|)$. Then, by compactness of $\bar S$, we can choose $\delta > 0$ such that 
$| v_i - v_{i+1} | > (1 - \varepsilon_4) \D_0(v_i, v_{i+1})$
with {$\varepsilon_4 = \varepsilon_3/10$}. If we call $(r_1, r_2, \ldots, r_k)$ the concatenation of the geodesics $r_1, r_2, \ldots, r_k$, we have 
	\begin{equation*}
	\sum_{i=0}^k {\left (\min_{r_i} f \right)^{-\beta} } 	
	| v_i - v_{i+1} | > \int_{(r_1, r_2, \ldots, r_k)} \frac{1}{f^\beta} - \frac{\varepsilon}{8\mu}.
	\end{equation*}
Then,
	\begin{align}
	& \mathbb P \left(  \sum_{i=0}^k n^\beta 
	D_{Q_n} (v_i, v_{i+1}) \leq \mu \mathcal D_{f, \beta}(x,y) - \frac{\varepsilon}{2}, \,  k \le K, \,  E_n
	\right) \nonumber \\
	\leq \text{ } 
	& \mathbb P \Bigg(  \sum_{i=0}^k   
	\mu M_{f,i}^{-\beta}   |v_i - v_{i+1} | - \frac{\varepsilon k }{8 K} \leq \mu \mathcal D_{f, \beta}(x,y) - \frac{\varepsilon}{2}, \, k\le K, \,  E_n, \, \bigcap_{i=0}^k B_i
	\Bigg)  + 	\sum_{i=0}^k \mathbb{P}(B_i^c) \nonumber \\
	\leq \text{ }
	& \mathbb P \left(  \mu \int_{(r_1, \dots, r_k)} \frac{1}{f^\beta} \leq \mu \mathcal D_{f, \beta}(x,y) - \frac{\varepsilon}{8},\,  k\le K,\, E_n,\, \bigcap_{i=0}^k B_i
	\right)
	+
	\sum_{i=0}^k \mathbb{P}(B_i^c)	\label{eq:proba_integral_cero}.
	\end{align}	 	
        Since
        \[
         \int_{(r_1, \ldots , r_k)} \frac{1}{f^\beta} \ge \mathcal D_{f, \beta}(x,y),
        \]
the first term in \eqref{eq:proba_integral_cero} is zero. Combining all these facts, we get 
  	\begin{align*}
  	\mathbb{P} \bigg (  n^\beta D_{Q_n}(x, y) & \leq \mu \mathcal D_{f, \beta}(x,y)  - \varepsilon \bigg )
  	\, \leq \, 
  	\mathbb P \left( k \ge K  \right)
  	+ 
  	\mathbb P \left( E_n^c \right)
  	+
  	\sum_{\substack{v_1, \ldots , v_k \in \mathcal V \\ |v_i - v_{i+1}| > \delta / 2}} \sum_{i=0}^k \mathbb{P} (B_i^c)\\
  	& \le \exp(-\constB(M_f n)^{\constC}) +{(\kappa_3 n)^2 \exp(-\constB (m_f n)^{\constC})} \\
  	& +  \exp(-\constN n)\\
  	& + (\kappa_3 n)^{K}( \exp \left( -\constB m_f^{\constC} n^{\constC} \right) + \exp \left( -\constL n^{\constC} \right))\\
  	& \le \exp(-\constQ n^{\constR}),
  	\end{align*}
  	for every $n\ge n_0$ if $\constQ$ and $n_0$ are chosen adequately. This concludes the proof of the lemma and Theorem \ref{teo:main_theorem_poisson}.
   	\end{proof}

\section{Manifolds}
\label{manifold}

We now consider the case in which the data is supported on a (possibly lower dimensional) manifold. We consider a manifold $\M$ that is the image of an isometric transformation from the closure of an open connect set of $\R^d$. The proof is based on the fact that a $d$-dimensional manifold is locally equivalent to $\R^d$ and that if in addition $\M$ is smooth enough, then geodesic and Euclidean distances are similar locally.

We consider $S \subset \R^d$ an open connected set and a diffeomorphism  $\phi : \bar S \mapsto \mathcal M := \phi(\bar S) \subset \R^D$, with $d < D$. 
Let $J_\phi(z) \in \R^{D \times d}$ be the Jacobian matrix of  $\phi$ defined by 
	\begin{equation*}
	\left( J_\phi (z) \right)_{ij} = \frac{\partial \phi_i}{\partial z_j} (z). 
	\end{equation*}	 
We assume that $\phi$ is an isometric transformation, i.e. for every $z \in S$ and $\mathbf v, \mathbf w\in \R^D$ tangent to $\M$ at $\phi(z)$ we have 
	\begin{equation*}
	\left( J_\phi(z) \mathbf v \right)^t
	\left( J_\phi(z) \mathbf w \right) 
	= 
	\mathbf v^t \mathbf w,
	\end{equation*}
which is equivalent to $J_{\phi}(z)^T J_\phi(z) = \mathbb I_d$. Here $\mathbb I_d$ is the identity matrix in $\R^{d\times d}$. If $\M$ is compact, then for every $\varepsilon_0 > 0$ there exists $\delta_0 > 0$ such that
	\begin{equation}
	(1 - \varepsilon_0) | \phi^{-1}(x) - \phi^{-1}(y) |
	<
	| x - y |
	<
	(1 + \varepsilon_0) | \phi^{-1}(x) - \phi^{-1}(y) |,
	\label{eq:manifold_local}  
	\end{equation}
	if $|x-y|<\delta_0$.

We first need to extend Lemma \ref{lema:spacing} to manifolds. The proof is straightforward and we omit it.

\begin{lem}
Assume $\M\subset \R^D$ is a $C^1$ $d$-dimensional manifold. Let $Q_n=\{q_1,\dots, q_n\}$ be independent random points with common density $f$. For $\alpha >1$ and $x, y \in \M$, let $(q_1, \dots, q_{k_n})$ be the minimizing path. Given $\delta > 0$, there exists a positive constant $\constX$ such that 
	\begin{equation*}
	\mathbb P \bigg ( \max_{i<k_n} | q_{i} - q_{i+1} | > \delta   \bigg ) \leq   \exp \left( - \constX n \right ).
	\end{equation*}	
	\label{lema:spacing.manifold}
	\end{lem}


\begin{proof}[Proof of Theorem \ref{teo:manifolds}] Given $Q_n$, we consider  $\tilde Q_n = \phi^{-1} (Q_n)$,  $\tilde x=\phi^{-1}(x)$, $\tilde y=\phi^{-1}(y)$. The points in $\tilde Q_n$ are independent, with common density $g: S \to \R_{\geq 0}$ given by 	
	\begin{equation*}
	g(z) = f( \phi(z) ) \sqrt{\det \left( J_\phi (z)^t J_\phi (z) \right) }  = f( \phi(z) ) .
	\end{equation*}	 
Given $\varepsilon_0 > 0$, let $\delta_0$ be as in \eqref{eq:manifold_local}. Then for every path $(q_1, q_2, \ldots , q_k)$ in $\M$ with $|q_i - q_{i+1}| < \delta_0$ we have
	\begin{equation*}
    (1-\varepsilon_0)^\alpha \sum_{i=1}^{k-1} | \tilde  q_{i+1} - \tilde q_i |^\alpha
	< 
	\sum_{i=1}^{k-1} | q_{i+1} - q_i |^\alpha 
	<
	(1+\varepsilon_0)^\alpha \sum_{i=1}^{k-1} | \tilde  q_{i+1} - \tilde q_i |^\alpha.
	\end{equation*}
Then, on the event $\{ n^\beta D_{\tilde Q_n} (\tilde x, \tilde y) < 2 \mu m_f^{-\beta} \D_0(x,y) \}$ we can choose $\varepsilon_0$ small enough to guarantee 
	\begin{equation*}
	\left | n^\beta D_{Q_n}(x,y) - n^\beta D_{\tilde Q_n}(\tilde x, \tilde y)   \right | < \frac{\varepsilon}{2}
	\end{equation*}
On the other hand, since $\phi$ is an isometry it holds 
\[
\mathcal D_{f,\beta}(x,y) = \inf_{\gamma \subset \M}\int_{\gamma} \frac{1}{f^\beta} =  \inf_{\sigma \subset S} \int_\sigma \frac{1}{g^\beta} = \mathcal D_{g,\beta}(\tilde x, \tilde y).
\] 
Finally,
	\begin{align}
	 \mathbb{P} \bigg(  \big | n^\beta D_{Q_n}(x, y) -  \mu \mathcal D_{f, \beta}(x,y)  \big | > \varepsilon \bigg)  
	 & \leq \, 
	  \mathbb{P} \bigg( \left | n^\beta D_{\tilde Q_n}(\tilde x, \tilde y ) - \mu \mathcal D_{g, \beta}(\tilde x , \tilde y )  \right | > \frac{\varepsilon}{2} \bigg) \label{eq:mani1} \\
	 & + \,
	 \mathbb P \bigg ( n^\beta D_{\tilde Q_n} (\tilde x, \tilde y) < 2 \mu f_{min}^{-\beta} \D_0(x,y)  \bigg). \label{eq:mani2} \\
	 & + \,
	 \mathbb P \bigg ( \max_{i<k_n} | q_{i} - q_{i+1} | > \delta_0  \bigg). \label{eq:mani3}
	\end{align}
	We bound \eqref{eq:mani1} by means of Theorem \ref{teo:main_theorem_poisson}. Lemma \ref{lema:bounds} is used to bound \eqref{eq:mani2} and Lemma \ref{lema:spacing.manifold} to bound \eqref{eq:mani3}, which concludes the proof.
	\end{proof}

%
%
%
%
%
%
%

\section{The arc length of geodesics}
\label{arclength}

{In this section we show a bound for the arc length of geodesics. We think this result is of independent interest. We then prove that miscroscopic geodesics converge to macroscopic ones. }
\begin{proof}[Proof of Proposition \ref{prop:bounded.arc.length}] 
Denote $r_n:=r_{Q_n,\alpha}(x,y)$ and $(q_1, \dots, q_{k_n}):=r_n$ the particles that form the minimizing path. Notice that $k_n$ is the number of particles in $r_n$.  From H\"older's inequality we have
	\begin{equation*}
	|\overline{r_n}| \leq  {k_n}^{(\alpha-1)/\alpha} D_{Q_n}(x,y)^{1/\alpha},
	\end{equation*}
	Then,
	\begin{align*}
	\mathbb{P} ( |\overline{r_n}|  & > \ell ) 
	 \leq 
	\mathbb{P} \left( n^\beta D_{Q_n}(x, y) \left( k_n n^{-1/d} \right)^{\alpha - 1} > \ell |\overline{r_n}|^{\alpha-1} \right) \nonumber \\
	& \leq 
	\mathbb{P} \left(  n^\beta D_{Q_n}(x, y) \left( \frac{k_n}{n^{1/d} |\overline{r_n}|} \right)^{\alpha-1}  > \ell \right) \nonumber \\
	& \leq 	
	\mathbb{P} \left( n^\beta D_{Q_n}(x, y) > 2\mu m_f^{-\beta}\D_0(x,y) \right) +
	\mathbb{P} \bigg( \frac{k_n}{n^{1/d} |\overline{r_n}|} > \left( {\ell}/{2\mu m_f^{-\beta}\D_0(x,y)} \right)^{1/(\alpha - 1)} \bigg). 
 	\end{align*}
The first term can be bounded by means of \eqref{eq:bounds_fmin}. To bound the second one we will show the existence of positive constants $\constH, \constI, \constJ$, with $\constJ$ depending only on $\delta$, such that
 	\begin{equation}
 	\mathbb P \left( \frac{k_n}{n^{1/d} |\overline{r_n}|} > \constH \right) \leq \constI \exp \left( - \constJ n^{1/d} \right).
 	\label{eq:bounded_KN}
 	\end{equation}
 	Then, if we take $\ell \ge  K \constH^{\alpha-1}$, we can conclude \eqref{eq:bounded_path}. The proof of \eqref{eq:bounded_KN} is similar to the one of Lemma 3 in \cite{HN}. Hereafter we include the adaptation of that proof to our context. Let us consider a covering $\mathcal C$ of $\R^d$ by closed cubes $C$ of edge size $\varepsilon = \varepsilon_0 n^{-1/d}$ and vertices in $\varepsilon_0 n^{-1/d} \mathbb Z^d$. That is, if $C \in \mathcal C$, then $C = z + [0, \varepsilon_0 n^{-1/d}]^d$ for some $z \in \varepsilon_0 n^{-1/d} \mathbb Z^d$. Let $\mathsf m_n = \# \{ C \in \mathcal C \colon C \cap \overline{r_n} \neq \emptyset \}$. 
	We say that two cubes (cells) $C$ and $C'$ are adjacent if they share a face and we denote that $C \sim C'$. We call $(C_{1}, \ldots, C_{m})$ a {\em path of cells} of length $m$ if $C_{j} \sim  C_{j+1}$ for every $j=1, \dots, m-1$. 
	Let us consider the event
	\begin{equation*}
	E_n^m =  \left \{ \text{There exist a path  } (C_{1}, \ldots, C_{m}) \text{ with } \# \bigcup_{j=1}^m C_{j}\cap Q_n \ge \frac{m}{2d} \right \}.
	\end{equation*}		
	Given $m$ cells $C_{1}, C_{2}, \ldots, C_{m}$, it is clear that $\# \bigcup_{j=1}^m C_{j}\cap Q_n$ is stochastically bounded by a random variable $V_m \sim \text{Poisson}(m \varepsilon_0^d M_f)$. By means of Chernoff bounds we get for $\theta \in \R$ that 
	\begin{align*}
	\mathbb{P} \left( \# \bigcup_{j=1}^m C_{j}\cap Q_n \geq \frac{m}{2d} \right) 
	& \leq
	\mathbb{P} \left( V_m \geq \frac{m}{2d} \right) \nonumber \\
	& = 
	\mathbb{P} \left( e^{ \theta V_m} \geq e^{\theta \frac{ m}{2d}} \right)  \nonumber \\
	& \leq
	\exp \left( -\theta\frac{ m}{2d} \right) \mathbb{E} \left ( e^{\theta V_m} \right ) \nonumber\\
	& = 
	\exp \left( - \theta\frac{ m}{2d} + m \varepsilon_0^d M_f (e^\theta - 1)  \right). 
	\end{align*}	 
	The total number of paths of cells of length $m$ with $x \in C_1$ is bounded above by $(2d)^{m}$. Then,
	\begin{equation*}
	\mathbb{P}(E_n^m) \leq \left ( 2d\exp \left( - \theta/{2d} \right) \exp \left( \varepsilon_0^d M_f (e^\theta - 1 ) \right) \right )^{m}.
	\end{equation*}		
 	Choosing $\theta>0$ such that $(2d) e^{-\theta / 2d} < e^{-1}/2$ and $\varepsilon_0 > 0$ such that $e^{\varepsilon_0^d M_f (e^\theta - 1 ) } < 2$, we obtain $\mathbb{P}(E_n^m) \leq e^{-m}$. Notice that any (particle) path from $x$ to $y$ must intersect at least $\kappa_4 \varepsilon_0^{-1} | x - y | n^{1/d}$ cells, for some geometric constant $\kappa_4 > 0$ that depends on $d$. Let 
 	\begin{equation*}
 	F_n = \left \{ \frac{\mathsf m_n}{2d} \leq k_n \right \} \subset \bigcup_{m \geq \frac{\kappa_4}{\varepsilon_0} | x - y | n^{1/d}} E_n^m.
 	\end{equation*}
 	Then,
 	\begin{equation*}
 	\mathbb{P}(F_n) \leq \sum_{m = \lfloor \frac{\kappa_4}{\varepsilon_0} |x - y| n^{1/d} \rfloor}^\infty \mathbb{P}\left( E_n^m \right) \leq e (1 - e^{-1})^{-1} e^{-\frac{\kappa_4}{\varepsilon_0} | x - y | n^{1/d}}.
 	\end{equation*}
 	Let $(C_1, C_2, \ldots , C_{\mathsf m_n})$ be the path of cells intersected by $\overline{r_n}$ sorted according to $r_n$. That is, let $(\gamma_n(t))_{0\le t \le |\overline{r_n}|}$ be the parametrization by arc length of the polygonal through $(q_1, \dots, q_{k_n})$ with $\gamma_n(0)=x$, $\gamma_n(|\overline{r_n}|)=y$. Then the cell-path is defined by
 	\[
 	C_1 \ni x, \quad \tau_0 =0,\qquad C_{j}\neq C_{j-1}, \quad C_j \ni \gamma({\tau_j}) \text{ with } \tau_j=\inf \{t>\tau_{j-1}  \colon \gamma(t) \notin C_{j-1} \}
 	\]
If $F_n^c$ occurs, then there are at least $\mathsf m_n/3d$ indices $i$ for which $d$ divides $i$, $i+d-1 < \mathsf m_n$ and $C_j\cap Q_n=\emptyset$ for all  $j$ with $i \leq j < i+d$. For each of these indices, there is a straight line that passes completely through $d$ adjacent cells $C_j$ and consequently crosses $d+1$ different hyperplanes of the grid $\varepsilon \Z^d$ . Using the Pigeonhole principle, we conclude that the straight line passes through two parallel hyperplanes separated by at least $\varepsilon$, that is, each line segment of $\overline{r_n}$ that passes completely through $d$ contiguous empty cells contributes at least $\varepsilon$ to the length $|\overline{r_n}|$. In other words, $|\overline{r_n}| \ge \frac{\mathsf m_n}{3d} \varepsilon$. Then
 	\begin{equation*}
 	k_n \leq \frac{\mathsf m_n}{2d} \leq \frac{3}{2\varepsilon_0} n^{1/d} |\overline{r_n}| \quad \text{in $F_n^c$}. 
 	\end{equation*}
Choosing
\begin{equation*}
	\constH \ge \frac{3}{2\varepsilon_0} 
	\quad , \quad
	\constI \ge e(1-e^{-1})^{-1} 
	\quad , \quad
	\constJ \le \frac{\kappa_4 \delta}{\varepsilon_0} \le \frac{\kappa_4}{\varepsilon_0} | x - y | 
	\end{equation*}	 	
 	we get \eqref{eq:bounded_KN}. We conclude the proof by taking $\constF(\delta) = \min \{ \constB(\delta) , \constJ(\delta) \}$ and from $\constC<1/d$.
\end{proof}

We are ready to prove Corollary \ref{cor:convrgence.of.geodesics}.

  	\begin{proof}[Proof of Corollary \ref{cor:convrgence.of.geodesics}]  We first need to define a topology in the space of curves contained in $S$. Let $\mathcal S$ be the set of continuous and rectifiable curves in $S$. For $\gamma, \gamma' \in \mathcal S$ define
	\begin{equation*}
	d_{\mathcal S}(\gamma, \gamma') = \min_{\substack{h \in P_\gamma \\ g \in P_{\gamma'}}}  \max_{t \in [0,1]} | h(t) - g(t) |.
 	\end{equation*}
 	Here $P_\gamma = \{h \colon [0,1] \to S, \, h \text{ is a parametrization of } \gamma \}$. Notice that $d_{\mathcal S}(\gamma, \gamma') < \delta$ implies  $\gamma \subset B(\gamma', \delta)$ and $\gamma' \subset B(\gamma,\delta)$. For every $\ell >0$, the set $\{ \gamma \in \mathcal S \colon |\gamma| \le \ell \}$ is compact with respect to this metric, \cite[Lemma 3]{myers1945arcs}. Observe also that the map $\gamma \mapsto \int_\gamma f^{-\beta}$ is continuous from $\mathcal S$ to $\R$.
 	
	For $\varepsilon_4 > 0$, we will see that the event $d_{\mathcal S}(\overline{r_n}, \gamma^\star) \geq \varepsilon_4$ occurs finitely many times. Since $\gamma^\star$ is the unique minimizer, there exist $\varepsilon_5 > 0$ such that 
  	\begin{equation*}
  	\int_{\gamma^\star} \frac{1}{f^\beta} + \varepsilon_5 < \inf_{d_{\mathcal S}(\gamma, \gamma^\star) \geq \varepsilon_4 } \int_\gamma \frac{1}{f^\beta}.
  	\end{equation*}
	Given $\varepsilon > 0$, by means of Theorem \ref{teo:main_theorem_poisson} with $S=B(\gamma, \delta)$ and the compactness of $\{|\gamma|<\ell^\star\}$ we get the existence of $\delta > 0$ such that for all $\gamma$ with $|\gamma|<\ell^\star$ 
  	\begin{equation}
	\mathbb{P} \left( \bigg | n^\beta D_{Q_n \cap B(\gamma,{\delta})}(x, y) - \mu  \int_{\gamma} \frac{1}{f^\beta} \bigg | > \varepsilon \right) < \exp({- \constY n^{\constC}}), \label{eq:esto_probamos_para_curvas}
  	\end{equation}
for some constant $\constY > 0$. Take $\varepsilon = \varepsilon_5 / 2$ and $\delta_5$ such that \eqref{eq:esto_probamos_para_curvas} holds. From the compactness of bounded sets of $\mathcal S$ we get the existence of a finite number of curves $\gamma^1, \gamma^2, \ldots , \gamma^m \in S \setminus \{ \gamma :  d_{\mathcal S}(\gamma, \gamma^\star) < \varepsilon_4 \} $ such that for every $\gamma \subset S$ continuous and rectifiable, with  arc length bounded by $\ell^\star$ and such that $d_{\mathcal S}(\gamma, \gamma^\star) \geq \varepsilon_4$, there exists $\gamma^j$ with $d_{\mathcal S}(\gamma, \gamma^j) < \min \{ \varepsilon_4, \delta_5 \}$. Then
	\begin{align*}
	\mathbb P \bigg ( d_{\mathcal S} (\overline{r_n} , \gamma^i) & < \delta \bigg )
	 \leq 
	\mathbb P \bigg ( n^\beta D_{Q_n}(x, y) = n^\beta D_{Q_n \cap B(\gamma^i,\delta_5)} ( x, y)  \bigg ) \\
	& \leq 
	\mathbb P \bigg ( \mu  \int_{\gamma^\star} \frac{1}{f^\beta} + \frac{\varepsilon_5}{2} > n^\beta D_{Q_n}(x, y) = n^\beta D_{Q_n \cap B(\gamma^i,\delta_5)} ( x, y) > \mu  \int_{\gamma^i} \frac{1}{f^\beta} - \frac{\varepsilon_5}{2}  \bigg )	\\
	& +
	\mathbb P \bigg ( \bigg | n^\beta D_{Q_n \cap B(\gamma^i,\delta_5) }(x, y) - \mu  \int_{\gamma^i} \frac{1}{f^\beta} \bigg | > \frac{\varepsilon_5}{2} \bigg) \\
	& +
	\mathbb P \bigg ( \bigg | n^\beta D_{Q_n}(x, y) - \mu  \int_{\gamma^\star} \frac{1}{f^\beta} \bigg | > \frac{\varepsilon_5}{2} \bigg ).
	\end{align*}
	The first term is zero and the last two terms decay exponentially fast as $n\to \infty$. By Borel-Cantelli's lemma, the event
	$\{ \overline{r_n} \subset S \setminus \{ \gamma : d_{\mathcal S}(\gamma, \gamma^\star) < \varepsilon_4 \} \}$ occurs finitely many times with probability one, as we wanted to prove.
	\end{proof}

\section{Restriction to nearest neighbors}
\label{implementation}
{In this section we prove that if we restrict ourselves to paths composed by $k$ nearest neighbors, the sample Fermat distance remains unchanged with high probability when $k \approx\log n$. This reduces the computational cost from $\mathcal O (n^3)$ to $\mathcal O(n^2 \log^2n)$.}
\begin{proof}[Proof of Proposition \ref{prop:k_nn}]
Recall that Given $k \geq 1$ and $q \in Q_n$, we denote the $k$-th nearest neighbor of $q$ by $q^{(k)}$ and we denote $\mathcal N_k(z) = \{ q^{(1)}, q^{(2)} , \ldots , q^{(k)} \}$ the set of $k$-nearest neighbors of $q$.

Given two points $z_1, z_2 \in S$ we define
		\begin{equation*}
		A^\alpha_{{z_1,z_2}} = \left \{  z \in S : | z_1 - z|^\alpha + | z_2 - z|^\alpha < | z_2 - z_1|^\alpha  \right \}.
		\end{equation*}
	There exists a constant $\delta>0$, that depends only on $\alpha$ such that $B((z_1+z_2)/2, \delta |\overline{z_1 z_2}| ) \subset A^\alpha_{z_1,z_2}.$ 
	Let $q_1, q_2 , \ldots, q_{k_n}$ be the optimal path and define
	\begin{equation*}
	k^\star = \min \left \{  k \in \N : q_{i+1}  \in \mathcal{N}_k(q_i) \text{ for all } i < k_n \right \}.
	\end{equation*}
	We need to prove
	\[
	 \P(k^\star > \constS \log ( n / \varepsilon ) + \constT) <\varepsilon.
	\]
	Notice that for every $1\le i \le k_n$, $A^\alpha_{{q_i,q_{i+1}}} \cap Q_n = \emptyset$ since if it is nonempty we can construct a path with lower cost than the minimizing path.
	For $k \in \N$, define the {random variable}
	\begin{align*}
	{\mathsf s}_k = \sup\Big \{ & s \colon \text{ there exists a ball $B_{s}$ with radius $s$ that contains at least $k$} \\ & \text{particles and another ball $B_{\delta s} \subset B_{s} $ with radius $\delta s$ and } B_{\delta s} \cap Q_n = \emptyset \Big \},
	\end{align*}
	and $A_k=\{\mathsf s_k > 0\}$. Here we use the convention $\sup \emptyset =0$.
	Since $q_{i+1} = (q_i)^{(k)}$ implies $A_k$, we have 
	\begin{equation}
	\Big \{ k^\star \geq k \Big \} \subset \bigcup_{j=k}^\infty A_j.
	\label{eq:Knn_max}
	\end{equation}
	Define
	\begin{equation*}
	\underline{\mathsf s} = \frac{1}{3}\left( \frac{k}{2M_fn} \right)^{1/d}, \quad  \bar{\mathsf{s}} =  2 \sqrt{d} \left( \frac{2k}{m_fn} \right)^{1/d},
	\end{equation*}
	Clearly $\underline{\mathsf s} < \bar{\mathsf{s}}$ and
	\begin{equation}
	\mathbb P (A_k) 
	= 
	\mathbb P (0 < \mathsf s_k < \underline{\mathsf s}) + 
	\mathbb P (\mathsf s_k > \bar{\mathsf{s}}) +
	\mathbb P (\mathsf s_k \in [\underline{\mathsf s}, \bar{\mathsf{s}}]).
	\label{eq:k_nn_partoProba}
	\end{equation}
	We proceed to bound each term in  \eqref{eq:k_nn_partoProba}. 
	\begin{align*}
	\mathbb P (0< \mathsf s_k < \underline{\mathsf s}) 
	& \leq \mathbb P \left( \text{$\exists$ a ball $B_{\underline{\mathsf s}} \subset S$ with radius $\underline{\mathsf s}$ with at least $k$ particles} \right) \\
	& \leq \mathbb P \left( \text{$\exists$ a cube $C_{2\underline{\mathsf s}} \subset S$ of edge size $2\underline{\mathsf s}$ with at least $k$ particles} \right).
	\end{align*}
	Consider the family $\mathcal C$ of cubes $C \subset \R^d$ with edge size $3\underline{{\mathsf s}}$ and vertices in $\underline{\mathsf s}\mathbb Z^d$. Notice that the number of elements in  $\mathcal C_S = \{C\cap S \colon C\in \mathcal C\}$ is bounded above by $\kappa^1_S n / k$, for some constant $\kappa^1_S$ that depends on the diameter of $S$. On the other hand, any cube with edge size $2\underline{\mathsf s}$ is strictly contained in a cube $C \in \mathcal C$. The number of particles in $C\cap S$ is a Poisson random variable with parameter bounded above by  $3^d\underline{\mathsf s}^d M_fn = k/2$. Then,
	\begin{equation*}
	\mathbb P (0  < \mathsf s_k < \underline{\mathsf s}) \leq 
	\kappa_S^1 \frac{n}{k} e^{-\theta_1 k},
	\end{equation*}	  
	for some positive constant $\theta_1$. Next,
	\begin{align*}
	\mathbb P \left(\mathsf s_k > \bar{\mathsf{s}} \right) 
	& \leq  \mathbb P \left( \text{$\exists$ a ball $B_{\bar{\mathsf s}} \subset S$ with radius $\bar{\mathsf s}$ with $k$ particles} \right) \\
	& \leq	\mathbb P \left( \text{$\exists$  a cube $C_{\bar{\mathsf{s}}/\sqrt{d}}\subset S$ with edge size $\bar{\mathsf{s}}/\sqrt{d}$ with at most $k$ particles}  \right).
	\end{align*}
	Now we consider the family $\mathcal C'$ of cubes $C\subset \R^d$ with edge size $\bar{\mathsf{s}}/(2\sqrt{d})$ and vertices in $(\bar{\mathsf s}/(2\sqrt{d})) \mathbb Z^d$. The number of elements in {$\mathcal C'_S = \{C\in \mathcal C' \colon C\subset S\}$} is bounded above by $\kappa_S^2 n / k$. If there is a cube  $C_{\bar{\mathsf{s}}/\sqrt{d}}$ with at most $k$ particles, then there is $C\in C'_S$ with at most $k$ particles. The number of particles in $C$ is Poisson with parameter at least $\bar{\mathsf{s}}^d m_f n / (2^d d^{d/2}) = 2k$. Then
	\begin{equation*}
	\mathbb P (\mathsf s_k  > \bar{\mathsf{s}}) \leq \kappa_S^2 \frac{n}{k} e^{-\theta_2 k},
	\end{equation*}
        for some positive constant $\theta_2$. Finally
	\begin{align*}
	\mathbb P \left(\underline{\delta \mathsf s} \le \mathsf s_k \le \bar{\mathsf{s}} \right) 
	& \leq
	\mathbb P \left( \text{$\exists$ ball $B_{\delta \underline{\mathsf s}} \subset S$ with radius $\delta \underline{\mathsf s}$ and } B_{\delta \underline{\mathsf s}}\cap Q_n =\emptyset \right) \\
	& \leq 
	\mathbb P \left( \text{$\exists$ cube $C_{\delta \underline{\mathsf s}/\sqrt{d}} \subset S$ with edge size $\delta \underline{\mathsf s} / \sqrt{d}$ and $C_{\delta \underline{\mathsf s}/\sqrt{d}}\cap Q_n=\emptyset$} \right).
	\end{align*}
	We proceed as before but now with the grid $(\delta \underline{\mathsf s} / 2\sqrt{d})\mathbb Z^d$. There is at most $\kappa_S^3 n/k$ cubes with vertices in the grid and nonempty intersection with $S$, the number of particles in a cube is Poisson with intensity no greater than $\underline{\mathsf s}^d M_fn / (2^d d^{d/2}) =  k / (2^{d+1} 3^d d^{d/2})$. Then, 
	\begin{equation*}
	\mathbb P \left(\underline{\mathsf s} \le  \mathsf s_k \le \bar{\mathsf{s}} \right) \leq \kappa_S^3 \frac{n}{k} e^{-\theta_3 k},
 	\end{equation*}
	with $\theta_3= (2^{d+1} 3^d d^{d/2})^{-1}$. We conclude that 
	\begin{equation*}
	\mathbb P (A_k) \leq \kappa_S \frac{n}{k} e^{-\theta k},
	\end{equation*}
	for $\theta = \min \{ \theta_1, \theta_2,  \theta_3 \}$ and $\kappa_S = \kappa_S^1 + \kappa_S^2 + \kappa_S^3$. By \eqref{eq:Knn_max} we get
	\begin{equation*}
	\mathbb{P} \left( k^\star \geq k \right) 
	\leq
	\sum_{j = k}^\infty \kappa_S \frac{n}{j} e^{-\theta j} \leq \kappa_S \frac{n}{k} (1 - e^{-\theta})^{-1} e^{-\theta k} 
	< \kappa_S n (1 - e^{-\theta})^{-1} e^{-\theta k}.
	\end{equation*}
	So, we can guarantee $\mathbb{P} \left( k^\star \geq k \right) < \varepsilon$ if 
	\begin{equation*}
	k {>} \frac{1}{\theta} \log \left( \frac{\kappa_S}{1-e^{-\theta}} \frac{n}{\varepsilon} \right).
	\end{equation*}
	This concludes the proof.
	  \end{proof}
\section*{Acknowledgments}
We want to thank Daniel Carando, Gabriel Larotonda, and Chuck Newman for enlightening conversations and acknowledge the team at Aristas, especially Yamila Barrera and Alfredo Umfurer, for useful discussions and implementation of algorithms. We also thank {Steven Damelin} and {Daniel Mckenzie} for private communications that helped us to clarify our respective contributions.
\bibliographystyle{plain}
\bibliography{biblio}

\end{document}